\documentclass[journal]{IEEEtran}
\ifCLASSINFOpdf
\else
\fi
\ifCLASSOPTIONcompsoc
  \usepackage[caption=false,font=normalsize,labelfont=sf,textfont=sf]{subfig}
\else
  \usepackage[caption=false,font=footnotesize]{subfig}
\fi
\usepackage{amsfonts,amsmath,amsthm}
\usepackage{graphicx}

\newcommand{\R}{\mathbb{R}}      
\newcommand*{\PlotPathPass}{.}

\newtheorem{theorem}{Theorem}[section]

\newtheorem{proposition}[theorem]{Proposition}
\newtheorem{definition}[theorem]{Definition}
\newtheorem{example}[theorem]{Example}
\newtheorem*{remark}{Remark}

\begin{document}
%
\title{Energy-Preserving and Passivity-Consistent Numerical Discretization of Port-Hamiltonian Systems}
%
%
%

\author{Elena Celledoni and Eirik Hoel H{\o}iseth\thanks{This work has received funding from the European Unions Horizon 2020
research and innovation programme under the Marie Sklodowska-Curie grant
agreement No. 691070, and from The Research Council of Norway.}\thanks{Elena Celledoni is currently with the Department of Mathematical Sciences at the Norwegian University of Science and Technology, Trondheim, 7491, Norway (e-mail: elena.celledoni@math.ntnu.no).}\thanks{Eirik H. H{\o}iseth (corresponding author) is currently with the Department of Mathematical Sciences at the Norwegian University of Science and Technology, Trondheim, 7491, Norway (e-mail: eirik.hoiseth@math.ntnu.no).}}

%
%

\markboth{Submitted to IEEE Transactions on Automatic Control in May 2017}%
{}
%



\maketitle

\begin{abstract}
In this paper we design discrete port-Hamiltonian systems systematically in two different ways, by applying discrete gradient methods and splitting methods respectively. The discrete port-Hamiltonian systems we get satisfy a discrete notion of passivity, which lets us, by choosing the input appropriately, make them globally asymptotically stable with respect to an equilibrium point. We test methods designed using the discrete gradient approach in numerical experiments, and the results are encouraging when compared to relevant existing integrators of identical order.
\end{abstract}

%
\IEEEpeerreviewmaketitle

%
%
%
%
\section{Introduction}
\IEEEPARstart{P}{ort-Hamiltonian}
systems are a recent and increasingly popular approach to modelling complex physical and engineering systems. This approach merges network theory with geometry and control. 

From network theory comes the concept of \emph{port-based modelling}, which allows for the modelling of complex systems, stretching over multiple physical domains. This is done by viewing the full system as a set of a (possibly large) number of simple ideal subsystems that are interconnected and communicate through the exchange of energy. Paynter pioneered this approach in \cite{paynter60}.

From geometric mechanics there is a focus on the underlying geometric structure of the system, see \cite{BBCM03}\cite{marsden99IMS}\cite{foundation}. Port-Hamiltonian systems represent a generalization of traditional Hamiltonian mechanics. Unlike traditional Hamiltonian mechanics, where the key geometry is that the phase space is endowed with a symplectic structure, the geometry of port-Hamiltonian systems comes from the interconnection structure of the system. The appropriate structure then appears to be a \emph{Dirac structure}, a generalization of both symplectic and Poisson structures that was first introduced in \cite{weinstein83}. Its use in port based modelling was first explored in \cite{vanderschaftymaschke95}\cite{dalsmo98ora}. An essential property of Dirac structures is that their appropriate composition again constitutes a Dirac structure. This ensures that interconnecting multiple port-Hamiltonian systems into a larger such system preserves this geometry.

Port-Hamiltonian systems can interact with their environment, and consequently the theory of control systems feature prominently.  
For our purposes the relevant example is interaction through inputs and outputs. Port-Hamiltonian systems can also be viewed as a technique for control design \cite{vanderschaft14}\cite{ortega99sop}\cite{dalsmo98ora}, e.g. by shaping the system energy or viewing controllers as virtual system components. A thorough introduction to port-Hamiltonian systems can be found in \cite{vanderschaft14}. 

In this paper we are concerned with the preservation of the remarkable properties of port-Hamiltonian system under numerical discretisation. We focus in particular on the energy balance and on the stability under interconnection. We will see that these properties are not automatically satisfied when replacing a continuous port-Hamiltonian system with its discrete counterpart obtained by applying a numerical discretisation method. And we propose two numerical approaches that will guarantee this preservation.

In geometric numerical integration, one seeks numerical integration methods preserving the structure of the flow one wishes to integrate \cite{hairer06gni}. For Hamiltonian mechanics, particularly for the unconstrained case where the configuration space is linear, there is a rich theory of structure preserving integrators: notably symplectic integrators \cite{vogelaere56}\cite{ruth83}\cite{feng85}\cite{sanzserna94}\cite{marsden01} and energy-preserving, symmetric integrators \cite{gonzalez96}\cite{QT1996}\cite{quispel08}\cite{Hairer2011}.  For port-Hamiltonian systems, structure preserving integration is far less explored. 

We restrict ourselves to the class of input-state-output port-Hamiltonian systems, and propose two approaches to construct discrete port-Hamiltonian systems. Our discrete models arise from the structure-preserving integration of their continuous counterparts. We analyse these methods, focusing in particular on discrete 
energy-preserving and passivity-preserving interconnection of simpler systems. The structure-preserving (and in particular passivity-preserving) integration of these systems is of interest both from a theoretical perspective and in engineering applications. See \cite{celledoni17} for an application of passivity preserving splitting methods to the control of marine vessels.

The structure of the paper is as follows. In Section \ref{contPHS} we give relevant background theory on continuous input-state-output port-Hamiltonian systems, and how they can be interconnected. In Section \ref{discretePHS} we consider the problem of numerically discretizing such port-Hamiltonian systems while preserving a discrete analogue of passivity. This reduces to energy preservation when the input is zero. Section \ref{Numexperiments} is devoted to numerical experiments. Finally we make some concluding remarks in Section \ref{concluding}. A higher order generalization of the method given in Section \ref{discretePHS} is derived in the Appendix.

\section{Background Theory}
\label{contPHS}
From the perspective of geometric mechanics, an input-state-output port-Hamiltonian system may be naturally introduced as a generalization of a traditional Hamiltonian mechanical system. In the absence of dissipative elements the following system of ordinary differential equations (ODEs) constitutes an input-state-output port-Hamiltonian system:
 
\begin{IEEEeqnarray}{rCl}
\label{PHS}  
\dot{x}&=&B(x)\, \nabla H(x)+ G(x)u, \nonumber\\
x(0)&=&x_0,\\
y&=&G(x)^T\nabla H (x), \nonumber
\end{IEEEeqnarray}
where $x\in\R^n$ is the state, $u\in\R^m$ the input and $y\in\R^m$ the output. Furthermore $B(x)$ is a skew-symmetric matrix (often, but not always, $B(x)$ defines a Poisson bracket), $H$ is the Hamiltonian function and $\nabla H (x)$ is the gradient of $H$ with respect to $x$. The input $u$ is given as a function of $t$, $x$ or $y$. We will usually take $u = u(y)$ reflecting the intuitive notion that the input often can only depend on the observable part of the system. 

The uncontrolled system, $\dot{x}=B(x)\, \nabla H(x)$, is assumed to have an isolated equilibrium point $x=x^*$. Since the change of coordinates $x \mapsto x-x^*$ will always move this equilibrium point to the origin, there is no loss of generality in taking $x^*=0$.

\subsection{Passivity}
Consider initially the general system of differential equations 
\begin{IEEEeqnarray}{rCl}
\label{gensys}
\dot{x} &=& f(x,u), \\
y &=& h(x), \nonumber
\end{IEEEeqnarray}
with state $x$, input $u$ and output $y$. $f: \R^n \times \R^m \mapsto \R^n$ is assumed to be locally Lipschitz with $f(0,0) = 0$ and $h : \R^n \mapsto \R^m$ continuous with $h(0) = 0$.

A common definition of passivity for such a system is the following from \cite[p. 236]{khalil02nls}:

\begin{definition}
\label{PassBasedControlDef}
The system \eqref{gensys} is passive if there exists a continuously differentiable positive semidefinite function $V(x)$, called the \emph{storage function}, such that
\begin{IEEEeqnarray}{c}
\label{passiveDef}
u^Ty \geq \dot{V} =  \nabla V(x)^T f(x,u), \ \ \forall (x,u) \in  \R^n \times \R^m.
\end{IEEEeqnarray}
\end{definition}
The system is said to be \emph{lossless} if $u^Ty = \dot{V}$. Integrating \eqref{passiveDef} we get the integral version of this passivity inequality
\begin{IEEEeqnarray}{rCl}
\label{passiveDefInt}
\left\langle y,u \right\rangle_{L^2} &:=& \int_0^ty^Tu\, ds \geq V(t)-V(0),\\
&&\forall \ \ t \geq 0, \ \ x(0) \in  \R^n, \ \ u(s): [0,t] \mapsto \R^m \,. \nonumber
\end{IEEEeqnarray}
We return to our port-Hamiltonian system \eqref{PHS}, which is of the format \eqref{gensys}. Differentiating the energy $H$ with respect to time, we obtain a differential equation for $H$
\begin{IEEEeqnarray}{c}
\label{eq:En1}
\dot{H}=y^Tu,
\end{IEEEeqnarray}
which states that the change in energy is equal to the work due to the external forces. This implies that the system \eqref{PHS} is passive, specifically lossless, with $H$ as the storage function. From the integral inequality \eqref{passiveDefInt} a system is passive with respect to the energy $H$ if it satisfies the inequality
\begin{IEEEeqnarray}{c}
\label{passivityPI}
\left\langle y,u \right\rangle_{L^2} \geq H(t)-H(0).
\end{IEEEeqnarray}
This means that such passive systems may consume and store energy, but are incapable of producing energy. For literature on the theory of passive systems see e.g. \cite{egeland2002mas}\cite{ortega98pce}\cite{khalil02nls}. An important consequence of this property is that if a system is passive it is possible to achieve asymptotic stability of the system by adding appropriate damping.

We will need the following two definitions from \cite{khalil02nls}: 

\begin{definition}
\label{def:zerostateobservable}
A system of the form (\ref{PHS}) is  \emph{zero-state observable} if no solution of $\dot{x}=B(x)\, \nabla H(x)$ can stay identically in the set $\{y(x)=0 \}$ other than the trivial solution $x(t)\equiv 0$.
\end{definition}

\begin{definition}
\label{def:radiallyunbounded}
A function $H:\R^n \mapsto \R$ is \emph{radially unbounded}  if $H(x)\rightarrow \infty$ as $\|x\|\rightarrow \infty$.
\end{definition}
Asymptotic stability, through the addition of appropriate damping, is given by the following theorem:
\begin{theorem}
\label{PassBasedControl}
If the passive system \eqref{PHS} has a radially unbounded positive definite Hamiltonian $H$,
and is zero-state observable, then the origin, $x=0$, can be globally stabilized by the input choice $u=-\phi(y)$ where $\phi: \R^m \mapsto \R^m$ is a locally Lipschitz function with the additional properties $\phi(0)=0$ and $y^T\phi(y)>0$ for all $y\neq 0$.
\end{theorem}
\begin{proof}
See \cite[p. 604]{khalil02nls}.
\end{proof}

\subsection{Interconnection}
Another important property of port-Hamiltonian systems is that they are stable under interconnection.
Given two port-Hamiltonian systems one can create a third one via a procedure called  {\it interconnection}, which we describe briefly in what follows.
Consider the two port-Hamiltonian systems

\begin{IEEEeqnarray}{rCl}
\label{PHS1}
\dot{x}&=&B(x)\, \nabla H(x)+ G(x)u,\nonumber \\
x(0)&=&x_0, \\
y&=&G(x)^T\nabla H (x), \nonumber
\end{IEEEeqnarray}
and
\begin{IEEEeqnarray}{rCl}
\label{PHS2}
\dot{\gamma}&=&B_c(\gamma)\, \nabla H_c(\gamma)+ G_c(\gamma)u_c, \nonumber \\
\gamma(0)&=&\gamma_0, \\
y_c&=&G_c(\gamma)^T\nabla H_c (\gamma). \nonumber
\end{IEEEeqnarray}
We join the two systems by {\it interconnection} by imposing
the energy balance condition
\begin{IEEEeqnarray}{c}
\label{interconnectioncondition}
y^Tu+y_c^Tu_c=0,
\end{IEEEeqnarray}
which states that energy flowing out of one system through the ports flows into the other. A simple way to satisfy \eqref{interconnectioncondition} is to take
\begin{IEEEeqnarray}{c}
\label{interconnection}
u=-y_c,\quad u_c=y.
\end{IEEEeqnarray}
Using \eqref{interconnection} we obtain the larger system
\begin{IEEEeqnarray}{rCl}
\label{interconnectedsystem}
\left[\begin{array}{c}
\dot{x}\\
\dot{\gamma}
\end{array}
\right] &=&C(x,\gamma)\, \left[ \begin{array}{c}
\nabla H (x)\\
\nabla H_c(\gamma)
\end{array}\right], \\ 
C(x,\gamma)&:=&\left[\begin{array}{cc}
B(x) & -G(x)G_c(\gamma)^T\\
G_c(\gamma)G(x)^T & B_c(\gamma)
\end{array} \right], \nonumber
\end{IEEEeqnarray}
with Hamiltonian $H(x)+H_c(\gamma)$. The obtained system is port-Hamiltonian. Because of this property one says that port-Hamiltonian systems are stable under interconnection. Usually the first system is given, and the second system is designed to control the first one. This means that one should design $H_c(\gamma)$ so that the new system is driven to the desired equilibrium state $x^*$, here again taken to be $0$ without loss of generality. We note that the Casimirs of the larger system are also of importance, for the purpose of control design. See for example \cite{vanderschaft14} for details. We also mention that if the skew-symmetric structure matrix $C(x,\gamma)$ satisfies also the Jacobi identity, then $C(x,\gamma)$ defines a Poisson bracket.


\subsection{Interconnection and Generalized Dirac Structures}\label{Dirac structures}

Let $\mathcal{Q}$ be a smooth manifold and let $T\mathcal{Q}$ and $T^*\mathcal{Q}$ be its tangent and cotangent bundle respectively. Consider the smooth vector bundle over $\mathcal{Q}$
$$
T\mathcal{Q}\oplus T^*\mathcal{Q},
$$
with fibres $T_x\mathcal{Q}\times T_x^*\mathcal{Q}$.

A generalised Dirac structure on $\mathcal{Q}$  is a vector subbundle 
$$
\mathcal{D}\subset T\mathcal{Q}\oplus T^*\mathcal{Q},\quad \mathrm{s.t.}\quad \mathcal{D}=\mathcal{D}^{\perp},
$$
where
\begin{IEEEeqnarray*}{rCl}
\mathcal{D}^{\perp}&:=&\{\, (e,f)\in T\mathcal{Q}\oplus T^*\mathcal{Q} \,|\,\, \langle e,f' \rangle+ \langle e',f\rangle=0,\,\, \\ 
&&\forall (e',f')\in \mathcal{D} \,\}, 
\end{IEEEeqnarray*}
and $\langle \cdot , \cdot \rangle$ is the duality pairing between $T_x\mathcal{Q}$ and $T_x^*\mathcal{Q}$.

If $\mathcal{F}$ is an $n$-dimensional vector space, it can be shown that $\mathcal{D}\subset \mathcal{F}\oplus\mathcal{F}^*$ is a (constant) Dirac structure on $\mathcal{F}$ if and only if

\begin{itemize}
\item $\langle e , f\rangle =0$ for all $(e,f)\in \mathcal{D}$. 
\item $\mathrm{dim}(\mathcal{D})=n$.
\end{itemize}

 Symplectic structures induce Dirac structures on $\mathcal{Q}$. Let us denote with  $\mathcal{X}_H$ the Hamiltonian vector field with Hamiltonian $H$ with respect to an  almost-symplectic structure $\Omega $ on $\mathcal{Q}$ ($\Omega$ is a nondegenerate two-form on $\mathcal{Q}$ which is not necessarily closed), 
and let
\begin{IEEEeqnarray}{rCl}
\mathcal{D}_\Omega&:=&\{\,( \mathcal{X}_H,dH)\in T\mathcal{Q}\oplus T^*\mathcal{Q}\,|\, \Omega(\mathcal{X}_H,\cdot)=dH, \nonumber\\
 &&H:\mathcal{Q}\rightarrow  \mathbb{R}\, \},  \label{DiracStexample}
\end{IEEEeqnarray}
then $\mathcal{D}_\Omega$ is a generalized Dirac structure. In this sense Dirac structures are generalisations of symplectic structures.

Let $\tilde{\Omega}$ be the two form associated with the skew-symmetric matrix $C$ in \eqref{interconnectedsystem}. The interconnection of the two port-Hamiltonian systems \eqref{PHS1} and \eqref{PHS2} under the condition \eqref{interconnectioncondition} gives a larger system with state variables $X=(x,\gamma)$, with energy $\tilde{H}=H+H_c$. The couple $(\dot{X},d\tilde{H})$ such that  $\dot{X}=C(X)\,\nabla \tilde{H}(X)$, belongs to the Dirac structure $\mathcal{D}_{\tilde{\Omega}}$ induced by $\tilde{\Omega}$.

\section{Discrete Port-Hamiltonian Systems and Discrete Passivity Based Control}
\label{discretePHS}

In this section we propose a definition of {\it discrete port-Hamiltonian systems}, see also \cite{aues13cio}.
For the numerical discretization of port-Hamiltonian systems we will focus on two important aspects: the preservation of a discrete energy balance equation and the stability under interconnection.

\subsection{Discrete Energy Balance}
To start we consider a general discrete system 
\begin{IEEEeqnarray}{c}
\label{disSys}
x_{n+1}=\Phi(x_n),
\end{IEEEeqnarray}
with the given initial state $x_0 \in \mathbb{R}^n$. A function $V: \mathbb{R}^n \mapsto \mathbb{R}$ is called a (discrete) Lyapunov function for \eqref{disSys} on a set $S \subset \mathbb{R}^n$ if it is continuous on $\mathbb{R}^n$ and $\Delta V(x) := V(\Phi(x))-V(x) \leq 0$ for all $x \in S$. We require the following discrete Invariance Principle from \cite[p. 9]{lasalle86}:

\begin{proposition}
\label{invariance}
If $V$ is Lyapunov function for \eqref{disSys} in $S \subset \mathbb{R}^n$ and the solution $x_n$ of \eqref{disSys} is in $S$ and bounded, then there is number $c$ such that $x_n \rightarrow M \cap V^{-1}(c) \neq \emptyset$ where $M$ is the largest positively invariant set contained in the set $E = \{x \in \mathbb{R}^n: \Delta V= 0 \} \cap \bar{S}$.
\end{proposition}

Now let us consider a consistent numerical integration method $x_{n+1}=\Phi_h(x_n,u)$ of \eqref{PHS} producing the approximations $x_n\approx x(t_n)$ for $t_n=n\, h$ with $h$ the step size of integration. Clearly the solution will depend on the choice of input function $u$ in \eqref{PHS}, which we here assume is given as a function of the output $y$, i.e. $u = u(y)$. 
\begin{definition}
Assume the method $x_{n+1}=\Phi_h(x_n,u)$ produces $m$ intermediate approximations of the output and of the input, $Y_n := [y_{n1},\hdots,y_{nm}]$ and $U_n := [u_{n1},\hdots,u_{nm}]$, with $u_{nj}:=u(y_{nj})$ and
$$\lim_{h\rightarrow 0}y_{nj} =y(x(t_n)), \quad \lim_{h\rightarrow 0} u_{nj} = u(y(x(t_n))),$$ 
$j = 1,\hdots,m$. We say the method satisfies a \emph{discrete energy balance equation} if there exist positive weights $b_j$ with $\displaystyle \sum_j b_j = 1$ such that
\begin{IEEEeqnarray}{rCl}
\Delta H(x_n) &:=& H(x_{n+1})-H(x_n)= \left\langle Y_n,U_n \right\rangle_{L^2_h} \nonumber \\
&:=& h\sum_{j=1}^m b_jy_{nj}^Tu_{nj}\, ,  \label{DEC}
\end{IEEEeqnarray}
holds for arbitrary $n$ and $h$.
\end{definition}
Note that $\displaystyle \lim_{h \to 0} \left\langle Y_n,U_n \right\rangle_{L^2_h} = \left\langle y,u \right\rangle_{L^2}$. This property will be used to prove a discrete analogue of Theorem~\ref{PassBasedControl}.

\begin{theorem}
\label{discPassBasedControl}
Suppose the continuous system \eqref{PHS} has a radially unbounded positive definite storage function $H$, and that the consistent numerical method $x_{n+1}=\Phi_h(x_n,u)$ for this system satisfies a discrete energy balance \eqref{DEC}. Furthermore assume that no solution sequence of $x_{n+1}=\Phi_h(x_n,0)$ gives zero output, i.e. $\{y_{nj}=0\}$ for all $n$, $j$, except the trivial solution, $x_n = 0$, for all $n$ (discrete zero-state observability). The origin, $x=0$, can then be globally stabilised with the choice of an appropriate control input 
$u=-\phi(y)$, where $\phi$ is a locally Lipschitz function such that $\phi(0)=0$ and $y^T\phi(y)>0$ for all $y\neq 0$.
\end{theorem}
\begin{proof}
From the discrete energy balance
\begin{IEEEeqnarray*}{rCl}
\Delta H(x_n) &=& \left\langle Y_n,U_n \right\rangle_{L^2_h} := h\sum_{j=1}^m b_jy_{nj}^Tu_{nj} \\
&=& -h\sum_{j=1}^m b_jy_{nj}^T\phi(y_{nj}) \leq 0 ,
\end{IEEEeqnarray*}
where the last inequality follows from the properties of $\phi$ and the positiveness of the weights $b_j$. Since $H$ is continuous on $\mathbb{R}^n$ it is a (discrete) Lyapunov function on $\mathbb{R}^n$ for the discrete method $x_{n+1}=\Phi_h(x_n,u)$. In addition, because $H$ is radially unbounded, it follows that all solutions of this discrete system are bounded.

Consequently from Proposition \ref{invariance} $x_n \rightarrow M \neq \emptyset$, where $M$ is the largest positively invariant set contained in the set $E = \{x \in \mathbb{R}^n: \Delta H(x) = 0\}$. Thus if $M = \{0\}$ the origin will be globally asymptotically stable.

Now, from the above calculations $\Delta H(x_n) = 0$ implies $y_{nj} = 0$ for all $j$, which from the properties of $\phi$ implies that $u_{nj} = 0$ for all $j$. The zero-state observability requirement yields $x_n = 0$ for $n \in \mathbb{N}$, and consequently $M = \{0\}$.
\end{proof}

\subsection{Discrete Gradient  Methods}

A discrete gradient $\bar{\nabla} H: \mathbb{R}^n\times \mathbb{R}^n \rightarrow \mathbb{R}^n$ is an approximation of the gradient of a function $H:\mathbb{R}^n\rightarrow \mathbb{R}$, satisfying the following two properties:
\begin{enumerate}
\item $\bar{\nabla} H(x,x')^T(x'-x)=H(x')-H(x)$,
\item $ \bar{\nabla} H(x,x)= \nabla H(x)$.
\end{enumerate}

We consider the following consistent numerical discretization of (\ref{PHS}):
\begin{IEEEeqnarray}{rCl}
\frac{x_{n+1}-x_n}{h}&=&\tilde{B}(x_n,x_{n+1})\bar{\nabla} H(x_n,x_{n+1})\nonumber \\
&& +\> \tilde{G}(x_n,x_{n+1})\,\tilde{u}_n,\label{disgradPHS}
\end{IEEEeqnarray}
where we define $\tilde{u}_n := u(\tilde{y}_n)$, and the discrete output $\tilde{y}_n$ is defined to be
\begin{IEEEeqnarray}{c}
\label{disOut}
\tilde{y}_n := \tilde{G}(x_n,x_{n+1})^T\bar{\nabla}H(x_n,x_{n+1}).
\end{IEEEeqnarray}
Here $\bar{\nabla} H(x_n,x_{n+1})$ is a discrete gradient, $\tilde{B}$ and $\tilde{G}$ depend on $x_n$ and $x_{n+1}$ continuously, and are consistent discretizations, e.g. $\displaystyle \lim_{h \to 0}\tilde{B}(x_n,x_{n+1}) = B(x_n)$ in the case of $B$, and, in addition, $\tilde{B}(x_n,x_{n+1})$ is assumed to be skew-symmetric. 

From the first property of discrete gradients
one easily verifies that a discrete energy balance equation is satisfied. In fact
\begin{IEEEeqnarray*}{rCl}
\IEEEeqnarraymulticol{3}{l}{
H(x_{n+1})-H(x_n)}\\*
&=&h\bar{\nabla} H(x_n,x_{n+1})^T\tilde{B}(x_n,x_{n+1})\bar{\nabla} H(x_n,x_{n+1})\\
&&+\> h\bar{\nabla} H(x_n,x_{n+1})^T\tilde{G}(x_n,x_{n+1})\tilde{u}_n \\
&=&h\tilde{y}_n^T\tilde{u}_n.
\end{IEEEeqnarray*}

It is easy to verify that the hypotesis of Theorem \ref{discPassBasedControl} hold for the discrete passive systems of the form \eqref{disgradPHS}-\eqref{disOut}, with $m=1$, $y_{n1} := \tilde{y}_n$, $u_{n1} := \tilde{u}_n$ and $b_1 = 1$.

\begin{example}
One choice for the discrete gradient is the averaged vector field (AVF) discrete gradient
\begin{IEEEeqnarray}{rCl}
\bar{\nabla} H(x_n,x_{n+1})&\equiv &\int_0^1\nabla H(\rho({\alpha}))\,d\alpha, \label{avfdisgrad} \\
\rho(\tau) &=& x_n(1-\tau)+x_{n+1}\tau, \nonumber
\end{IEEEeqnarray}
Note that we have
\begin{IEEEeqnarray*}{rCl}
H(x_{n+1})-H(x_n)&=&\int_0^1\nabla H(\rho({\alpha}))^T\dot{\rho}(\alpha)\,d\alpha \\*
&=&\int_0^1\nabla H(p({\alpha}))^T\, d\alpha\,\,(x_{n+1}-x_n) \\
&=& \bar{\nabla} H(x_n,x_{n+1})^T \,(x_{n+1}-x_n), \\
\bar{\nabla} H(x_n,x_{n})&=&\int_0^1\nabla H(x_n)\,d\alpha  =\nabla H(x_n),
\end{IEEEeqnarray*}
so this is a discrete gradient.
If in addition we use the value at the midpoint $x_{n+\frac{1}{2}}=\frac{x_n+x_{n+1}}{2}$ to approximate $B$ and $G$, i.e. 
\[
\tilde{B}(x_n,x_{n+1}) = B\left(x_{n+\frac{1}{2}}\right), \ \ \tilde{G}(x_n,x_{n+1}) = G\left(x_{n+\frac{1}{2}}\right),
\] 
we get a second order method for \eqref{PHS} of the format \eqref{disgradPHS}. 
\begin{IEEEeqnarray}{rCl}
\frac{x_{n+1}-x_n}{h}&=& B\left(x_{n+\frac{1}{2}}\right)\int_0^1\nabla H(\rho({\alpha}))\,d\alpha \nonumber \\*
&& +\>G\left(x_{n+\frac{1}{2}}\right)\,\tilde{u}_n, \label{avfPHS}
\end{IEEEeqnarray}
\end{example}

\begin{remark} 
If $H$ is a polynomial of the components of $x$, then the integral in \eqref{avfPHS} can be explicitly computed. In particular for quadratic $H$ and linear $\nabla H$, one finds that (\ref{avfPHS}) coincides with the midpoint method. This explains the behaviour observed in \cite{aues13cio}, where the authors show that the midpoint method is energy-preserving for linear port-Hamiltonian systems. This property of the midpoint rule ceases to hold if $H$ is a polynomial function of higher degree than quadratic, see numerical experiments in Section~\ref{controlled pendulum}. Generalisations of \eqref{avfPHS} to higher order can be easily obtained using the ideas of \cite{Hairer2011}. See the Appendix for details. 
\end{remark}

\subsection{Interconnection and preservation of the generalised Dirac structure}
We consider the interconnection of two discrete port-Hamiltonian systems of the form \eqref{disgradPHS}-\eqref{disOut}

\begin{IEEEeqnarray*}{rCl}
\frac{x_{n+1}-x_n}{h}&=&B(x_{n+\frac{1}{2}})\, \bar{\nabla} H(x_n,x_{n+1})+ G(x_{n+\frac{1}{2}})u_{n+\frac{1}{2}},\\
y_{n+\frac{1}{2}}&=&G(x_{n+\frac{1}{2}})^T\bar{\nabla} H(x_n,x_{n+1}),
\end{IEEEeqnarray*}
\begin{IEEEeqnarray*}{rCl}
\frac{\gamma_{n+1}-\gamma_n}{h}&=&B_c(\gamma_{n+\frac{1}{2}})\, \bar{\nabla} H_c(\gamma_n,\gamma_{n+1})+ G_c(\gamma_{n+\frac{1}{2}})u_{c,n+\frac{1}{2}},\\
y_{c,n+\frac{1}{2}}&=&G_c(\gamma_{n+\frac{1}{2}})^T\bar{\nabla} H_c(\gamma_n,\gamma_{n+1}),
\end{IEEEeqnarray*}
under the  interconnection condition
$$y_{n+\frac{1}{2}}^Tu_{n+\frac{1}{2}}+y_{c,n+\frac{1}{2}}^Tu_{c,n+\frac{1}{2}}=0,$$
which we satisfy by imposing $u_{n+\frac{1}{2}}:=-y_{c,n+\frac{1}{2}}$, $u_{c,n+\frac{1}{2}}:=y_{n+\frac{1}{2}}$.
We obtain a larger discrete system 
$$
\left[\begin{array}{c}
\frac{x_{n+1}-x_n}{h}\\
\frac{\gamma_{n+1}-\gamma_n}{h}
\end{array}
\right] =C(x_{n+\frac{1}{2}},\gamma_{n+\frac{1}{2}})\, \left[ \begin{array}{c}
\bar{\nabla} H(x_n,x_{n+1})\\
\bar{\nabla} H_c(\gamma_n,\gamma_{n+1})
\end{array}\right],
$$
with $C(x,\gamma)$ given in \eqref{interconnectedsystem}. By the skew-symmetry of $C(x,\gamma)$, and the properties of discrete gradients, the obtained discrete system preserves the energy $\tilde{H}(x,\gamma)=H(x)+H_c(\gamma)$. In fact,
\begin{IEEEeqnarray*}{rCl}
\IEEEeqnarraymulticol{3}{l}{
\tilde{H}(x_{n+1},\gamma_{n+1})-\tilde{H}(x_n,\gamma_n)=\bar{\nabla} H(x_n,x_{n+1})^T(x_{n+1}-x_n)} \\*
&&+\> \bar{\nabla} H_c(\gamma_n,\gamma_{n+1})^T(\gamma_{n+1}-\gamma_n)=0.
\end{IEEEeqnarray*}
Using the definition~\eqref{DiracStexample}, we consider the constant Dirac structures $\mathcal{D}_{\tilde{\Omega}_{n}}$, where $\tilde{\Omega}_{n}$ is the two-form associated with the skew-symmetric matrix $C(x_{n+\frac{1}{2}},\gamma_{n+\frac{1}{2}})$. The couples of (time) discrete vector field and discrete gradient obtained by interconnection of the two discrete port-Hamiltonian systems and given by
$$\left(\left[\begin{array}{c}
\frac{x_{n+1}-x_n}{h}\\
\frac{\gamma_{n+1}-\gamma_n}{h}
\end{array}
\right],
\left[ \begin{array}{c}
\bar{\nabla} H(x_n,x_{n+1})\\
\bar{\nabla} H_c(\gamma_n,\gamma_{n+1})
\end{array}\right]
 \right), \quad  n=0,1,2,\dots$$
 belong to $\mathcal{D}_{\tilde{\Omega}_{n}}$ for all $n$. We can view $\mathcal{D}_{\tilde{\Omega}_{n}}$ for $n=0,1,2,\dots$ as a time-discrete  approximation of the Dirac structure $\mathcal{D}_{\tilde{\Omega}}$ considered at the end of Section~\ref{Dirac structures}.


\subsection{Splitting Methods}

We can also consider a splitting method. Assume the skew-symmetric matrix $B$ in \eqref{PHS} permits the splitting $B(x) = B_1(x) + B_2(x)$ where $B_1(x)$ and  $B_2(x)$ are again both skew-symmetric. Using this matrix splitting to split the vector field of \eqref{PHS}, pushing the control part into the second system, we have

\begin{IEEEeqnarray}{rCl}
S_1&:&\,\left\{\begin{array}{rcl}
\dot{x}&=&B_1(x)\, \nabla H(x)\, ,
\end{array}\right. \label{split} \\
S_2&:&\,\left\{\begin{array}{rcl}
\dot{x}&=&B_2(x)\, \nabla H(x) + G(x)u(y)\, .\\
\end{array}\right. \nonumber
\end{IEEEeqnarray}
with the normal output $y=G(x)^T\nabla H (x)$. 
Let the flow maps that advance the system some time $t$ forward along $S_1$ and $S_2$ be denoted $\Phi_{t}^{[S_1]}$ and $\Phi_{t}^{[S_2]}$ respectively. Now suppose we apply a splitting method 

\begin{IEEEeqnarray}{rCl}
x_{n+1} &=& \Phi_{a_1h}^{[S_2]}\circ \Phi_{b_1h}^{[S_1]}\circ \Phi_{a_2h}^{[S_2]}
\circ \cdots \circ \Phi_{a_{m+1}h}^{[S_2]}\nonumber \\
&&\circ \> \cdots \circ \Phi_{b_1h}^{[S_1]}\circ \Phi_{a_1h}^{[S_2]}\,
\left(x_n\right), \label{splittingGeneral}
\end{IEEEeqnarray}
to the full system \eqref{PHS}. Here we assume that all coefficients $a_i$ and $b_i$ are non-negative, and that the method is consistent, i.e.
\[
2\sum_{i=1}^{m} a_i + a_{m+1} = 2\sum_{i=1}^m b_i = 1.
\]
This implies that the method has a well defined numerical flow $\Phi_{\tau}$ with the property $\Phi_{2h}(x_n) = x_{n+1}$. This limits us to second order methods, as higher order splitting methods (with real coefficients) must have some stricly negative coefficients \cite{blanes05}.

\begin{theorem}
Let the splitting method \eqref{splittingGeneral} be applied to the splitting \eqref{split} of a system \eqref{PHS} with radially unbounded, continuous and positive definite storage function $H$. If no solution of $S_2$ with $u=0$ can stay in the set $y=0$ other than the trivial solution $x=0$, then the origin $x=0$ for the full system \eqref{PHS} can be globally stabilised with the choice of an appropriate control input $u(y)=-\phi(y)$. Here $\phi$ is a locally Lipschitz function such that $\phi(0)=0$ and $y^T\phi(y)>0$ for all $y\neq 0$.
\end{theorem}
\begin{proof}
Consider an arbitrary step from $x_n$ to $x_{n+1}$ for $n \in \mathbb{N}$. We apply $4m+1$ numerical flows alternating between $\Phi_{t}^{[S_1]}$ and $\Phi_{t}^{[S_2]}$. Let $\tilde{x}_k$ be the point we have reached after applying $k$ of these flows, e.g. $\tilde{x}_0 = x_n$ and $\tilde{x}_{4m+1} = x_{n+1}$. From \eqref{splittingGeneral} it is clear that on $[\tilde{x}_k,\tilde{x}_{k+1}]$ we are flowing along $\Phi_{t}^{[S_1]}$ if $k$ is odd and $\Phi_{t}^{[S_2]}$ if $k$ is even. It is also clear that for the system $S_1$, $\dot{H} = 0$ and for $S_2$, $\dot{H} = -y^T\phi(y) \leq 0$. Consequently $H(\tilde{x}_{k+1})-H(\tilde{x}_{k}) = 0$ if $k$ is odd and $H(\tilde{x}_{k+1})-H(\tilde{x}_{k}) \leq 0$ if $k$ is even. Thus
\begin{IEEEeqnarray*}{rCl}
\Delta H(x_n) &:=& H(x_{n+1})-H(x_{n}) := H(\tilde{x}_{4m+1})-H(\tilde{x}_{0}) \\*
&=& \sum_{k=0}^{4m} H(\tilde{x}_{k+1})-H(\tilde{x}_{k}) \leq 0.
\end{IEEEeqnarray*}
From an identical argument as in Theorem \ref{discPassBasedControl} $x_n \rightarrow M \neq \emptyset$, where $M$ is the largest positively invariant set contained in the set $E = \{x \in \mathbb{R}^n: \Delta H(x) = 0\}$. As before the origin will therefore be globally asymptotically stable if $M = \{0\}$.

Now, from the above calculations $\Delta H(x_n) = 0$ implies $y = 0$ while flowing along $S_2$, i.e. $[\tilde{x}_k,\tilde{x}_{k+1}]$ with $k$ even. From the properties of $\phi$ this implies that $u = 0$ here. The zero-state observability requirement then yields $x = 0$ here, which means $\tilde{x}_k = 0$ and thus $x_n = 0$ for $n \in \mathbb{N}$. Consequently $M = \{0\}$.
\end{proof}

\subsection{Discrete Energy Balance and Runge-Kutta Methods}

It can be easily shown that if the Hamiltonian is a polynomial function of the components of $x$, and the structure matrix $B$ does not depend on $x$, then applying the method \eqref{avfPHS} to \eqref{PHS} results in a Runge-Kutta method, see \cite{celledoni09epr}. This shows that if we restrict to polynomial Hamiltonian functions there exist Runge-Kutta methods which satisfy a discrete energy balance equation. A concrete example is given by the midpoint method applied to problems with constant $B$ and quadratic Hamiltonians, resulting in a linear port-Hamiltonian system. See Section~IIIB in \cite{aues13cio}.

However, without such restrictions, this is not possible. 
\begin{proposition}
No Runge-Kutta method satisfies \eqref{DEC} for general Hamiltonian functions $H$.
\end{proposition}
\begin{proof}
The proof is very similar to the proof of Proposition~4 in \cite{celledoni09epr}. Consider a system of type \eqref{PHS} with $x(t)=[q(t),p(t)]^T$ with the degenerate Hamiltonian function $H=p-F(q)$ and input from a derivative controller $u=[0,\bar{u}\dot{q}]^T$ ($\bar{u}$ the constant controller gain). We define $f(q):=\frac{\partial F}{\partial q}$. Let $B$ be the constant $2\times 2$ Darboux matrix, and $G=I$ the $2\times 2$ identity matrix. 
The equations \eqref{PHS} for this system are 
\begin{IEEEeqnarray*}{rCl}
\dot{q}&=&1,\\
\dot{p}&=&f(q)+\bar{u},\\
y&=&\nabla H,
\end{IEEEeqnarray*}
where $y$ is the output.
All B-series methods (including all Runge-Kutta methods)  over one step with initial condition  $q_0$ give $q_1=q_0+h$. Energy consistency according to 
\eqref{DEC} requires
\begin{IEEEeqnarray*}{c}
H(q_1,p_1)-H(q_0,p_0)=\left\langle Y_n,U_n \right\rangle_{L^2_h} \approx\int_0^h y^Tu\, ds=h\bar{u},
\end{IEEEeqnarray*}
and we observe that any consistent approximation $\left\langle Y_n,U_n \right\rangle_{L^2_h} \approx\int_0^h y^Tu\, ds=h\bar{u}$ would reproduce this integral exactly.
From this we get
\begin{IEEEeqnarray*}{rCl}
p_1&=&p_0+(F(q_1)-F(q_0))+\left\langle Y_n,U_n \right\rangle_{L^2_h} \\*
&=&p_0+\int_0^hf(q)\, ds+h\bar{u}.
\end{IEEEeqnarray*}
On the other hand a Runge-Kutta method would give an approximation $p_1$ in the form
$$p_1=p_0+h\sum_{i=1}^sb_if(c_ih,q_0+c_ih)+h\bar{u}.$$
This leads to the condition
$$h\sum_{i=1}^sb_if(c_ih,q_0+c_ih)=F(q_1)-F(q_0)=\int_0^hf(q)\, ds,$$
for the Runge-Kutta method, which can be satisfied for an arbitrary $F$ only if
 all quadrature conditions 
 $$
 \sum_{i=1}^sb_ic_i^{k-1}=\frac{1}{k}, \quad k=1,2,\dots
 $$
 are satisfied by the Runge-Kutta method.
\end{proof}

\section{Numerical Experiments}
\label{Numexperiments}
The numerical experiments focus on the introduced discrete gradient methods. See \cite{celledoni17} for an application of passivity preserving splitting methods.

\subsection{Controlled Rigid Body}
In the first numerical experiment we illustrate the preservation of the discrete energy balance equation for the method \eqref{avfPHS}, and see how the method achieves a correct energy exchange between external power and internal energy.
The test problem is a controlled rigid body spinning around its center of mass. The kinetic energy is
 $$H({\omega},q)=\frac{1}{2}(I_1\,\omega_1^2+I_2\,\omega_2^2+I_3\,\omega_3^2)+\frac{1}{2}q^Tq,$$
 with ${\omega}\in \mathbb{R}^3$ the angular velocity and $q\in \mathbb{R}^4$ the unit quaternion representing the attitude rotation of the body,
 and the equations are
\begin{IEEEeqnarray*}{rCl}
 \mathbb{I}\dot{\omega}&=&-\hat {\omega}\,\, \nabla_{\omega} H ({\omega},q)+u, \\\ 
 u&:=&-K_d\nabla_{\omega} H (\omega,q)-K_p\nabla_q H(\omega,q),\\
 \dot{q}&=&\left[ \begin{array}{cc}
 0 & \mathbf{0}^T\\
 \mathbf{0} & \hat{\omega} 
 \end{array}\right]\, \nabla_q H(\omega,q),
\end{IEEEeqnarray*}
with output
 $$y=\nabla H (\omega,q).$$
 Here $K_d$ is $3\times 3$ diagonal and $K_p$ is $3\times 4$.
 The energy balance equation reads
 \begin{IEEEeqnarray*}{rCl}
 \frac{d \, H}{dt}&=& \nabla_{\omega} H^T\dot{\omega}+\nabla_q H^T\dot{q}
 \\
  &=&\nabla_{\omega} H^T\mathbb{I}^{-1}u \\
  &=&-\nabla_{\omega} H^T\mathbb{I}^{-1}K_d\nabla_{\omega} H-\nabla_{\omega} H^T\mathbb{I}^{-1}K_p\nabla_q H.
  \end{IEEEeqnarray*}
 We apply the method \eqref{disgradPHS}-\eqref{disOut},
 and we obtain
the discrete energy balance equation 

\begin{IEEEeqnarray}{rCl}
H(\omega_N,q_N)-H(\omega_0,q_0)&=&-h\sum_{n=0}^{N-1} \bar{\nabla}_{\omega} H_n^T\,\mathbb{I}^{-1}K_d\,\bar{\nabla} _{\omega}H_n \nonumber \\*
&&-\> h\sum_{n=0}^{N-1} \bar{\nabla}_{\omega} H_n^T\,\mathbb{I}^{-1}K_p\,\bar{\nabla}_{q} H_n, \nonumber \\*
&&\label{discreteEnBalRB}
\end{IEEEeqnarray}
with
\begin{IEEEeqnarray*}{rCl}
 \bar{\nabla}_{\omega} H_n&:=& \bar{\nabla}_{\omega} H((\omega_{n},q_n),(\omega_{n+1},q_{n+1})), \\*  \bar{\nabla}_{q} H_n&:=& \bar{\nabla}_{q} H((\omega_{n},q_n),(\omega_{n+1},q_{n+1})).
\end{IEEEeqnarray*}

 In Fig.~\ref{CRB1} we plot (in red) separately the discrete external power $A_{ext}$ (i.e. minus the right hand side of the discrete energy balance equation), and the difference in the Hamiltonian $H(t_n) - H(t_0)$ in blue, (i.e. the left hand side of the discrete energy balance equation). We obtain the expected energy exchange. In Fig.~\ref{CRB2} we show that indeed the sum of these two energies is zero to machine precision. The inertia matrix is $\mathbb{I}=\mathrm{diag}(1,2,3)$, $K_d=\mathrm{diag}(3,4,5)$, $K_p=\left[ \mathrm{diag}(3,5,6),\mathbf{1} \right]$. Here $\mathbf{1}\in\mathbb{R}^3$ is the vector with all components equal to $1$.

\begin{figure}[!t]
 \begin{center}
\subfloat[]{\includegraphics[width=0.47\linewidth]{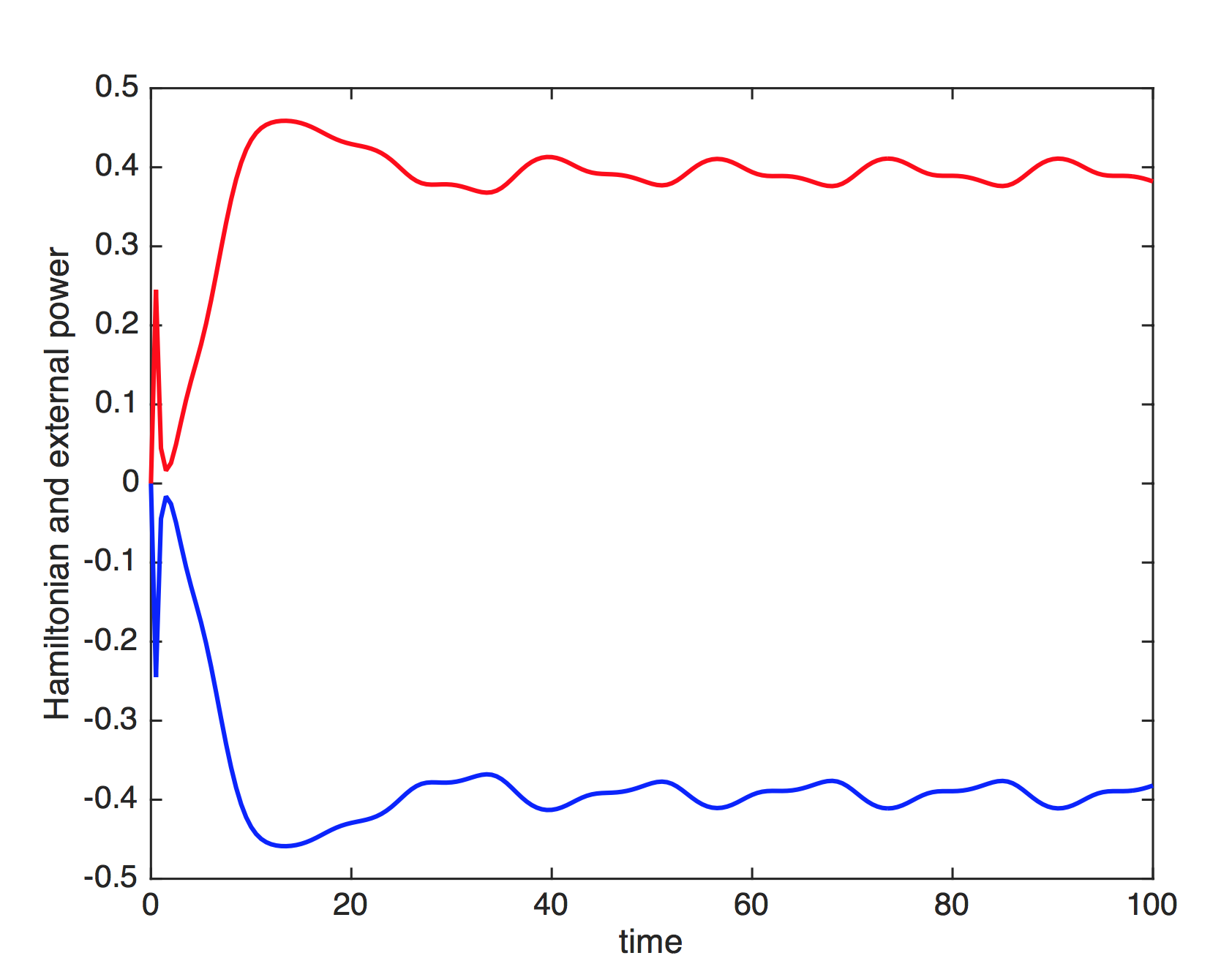}\label{CRB1}} 
\subfloat[]{\includegraphics[width=0.47\linewidth]{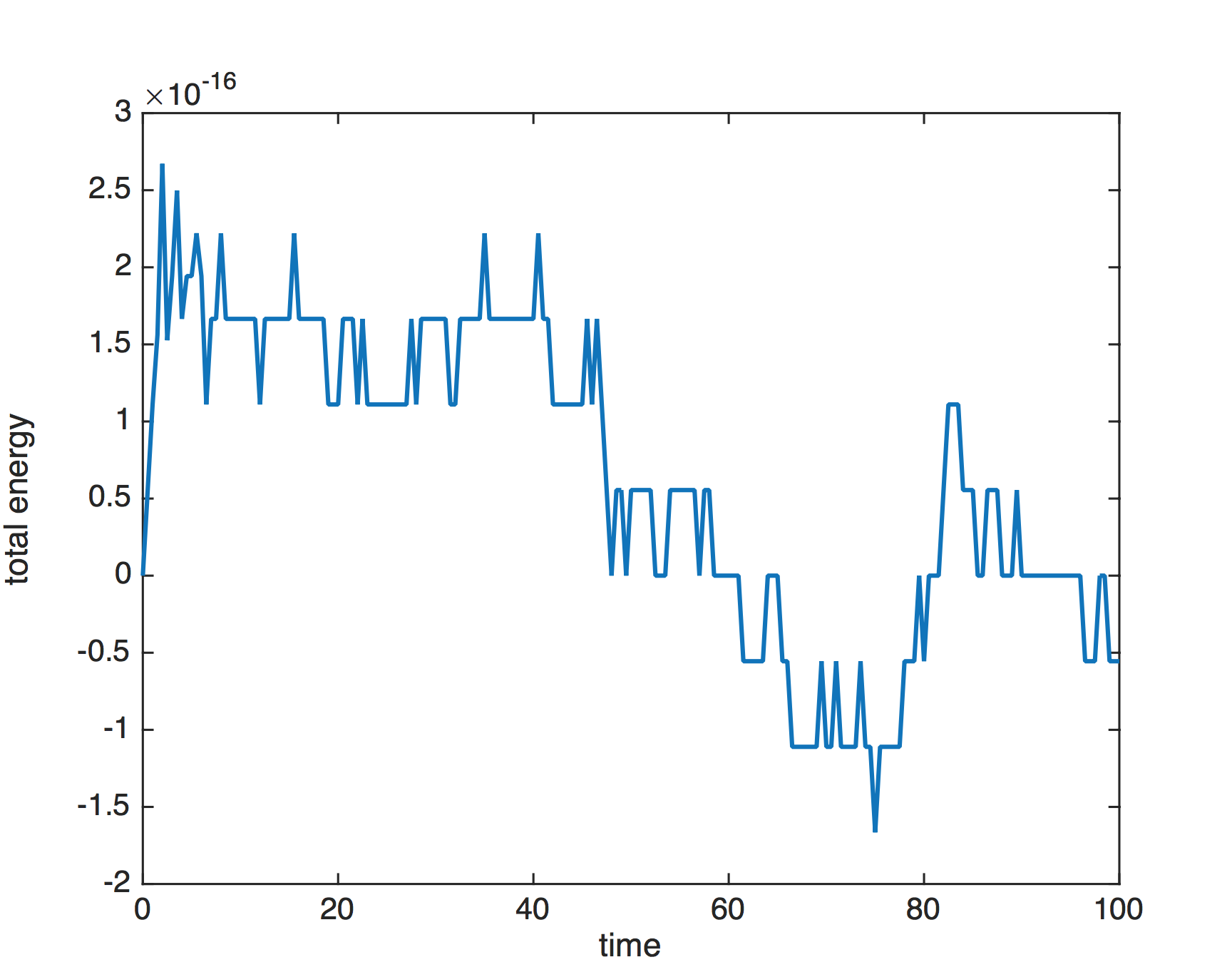}\label{CRB2}}
\end{center}
\caption{ (a) Plot of the evolution of the discrete external power $A_{ext}$ (red), and the difference in the Hamiltonian $H(t_n)-H(t_0)$ (blue). (b) The sum of these two energies. Step size $h=0.5$. \label{CRB}}

\end{figure}

\subsection{Controlled Pendulum}\label{controlled pendulum}
For a second experiment consider the simple pendulum, with a small non-linear controller term for the momentum. The system has the format \eqref{PHS} with $ x = [q,p]^T \in \mathbb{R}^2$ and
\begin{IEEEeqnarray}{rCl}
H(q,p) &=& \frac{1}{2}p^2 +1 - \cos q, \nonumber \\
B(x) &=& J, \nonumber  \\
G(x) &=& [0,1]^T, \label{pend}  \\
u &=& \phi(y) = -0.01\arctan y. \nonumber 
\end{IEEEeqnarray}
Using the theory from Theorem \ref{PassBasedControl}, one can show that this system will converge from almost every initial condition to the stable equilibrium $p=0$, $q = n2\pi$ for some integer $n$. Note that all these values of $q$ correspond to the same physical position. The system also has an unstable equilibrium at $p=0$, $q = (2n+1)\pi$.

In Fig. \ref{pendAng}, \ref{pendU}, and \ref{pendHam} we compare the evolution of the position, the absolute error in the Hamiltonian, and the input $u$ respectively for method \eqref{avfPHS}, the implicit midpoint method, the averaged vector field method, and the improved Euler method, all second order. The initial state is $x(0) = [2.8,1.4]^T$. We observe that the AVF method and the method \eqref{avfPHS} (which is based on this), outperforms the implicit midpoint, not to mention the improved Euler method, for the choice of step size, $h=0.5$. 

In particular, in Fig. \ref{pendAng}, the two latter methods give the wrong number of full rotations, $n$, of the pendulum before it starts to converge towards the stable equilibrium. Consequently the input signal as shown in Fig. \ref{pendU} is also qualitatively wrong for these methods. In contrast, both the AVF method and the method \eqref{avfPHS} produce results which are difficult to distinguish from the exact solution on the scale shown. The AVF method and the method \eqref{avfPHS} are seen to have comparable energy preservation in Fig. \ref{pendHam}, which is superior to the implicit midpoint and the improved Euler method.

\begin{figure}[!t]
 \begin{center}
\includegraphics[width=1\linewidth]{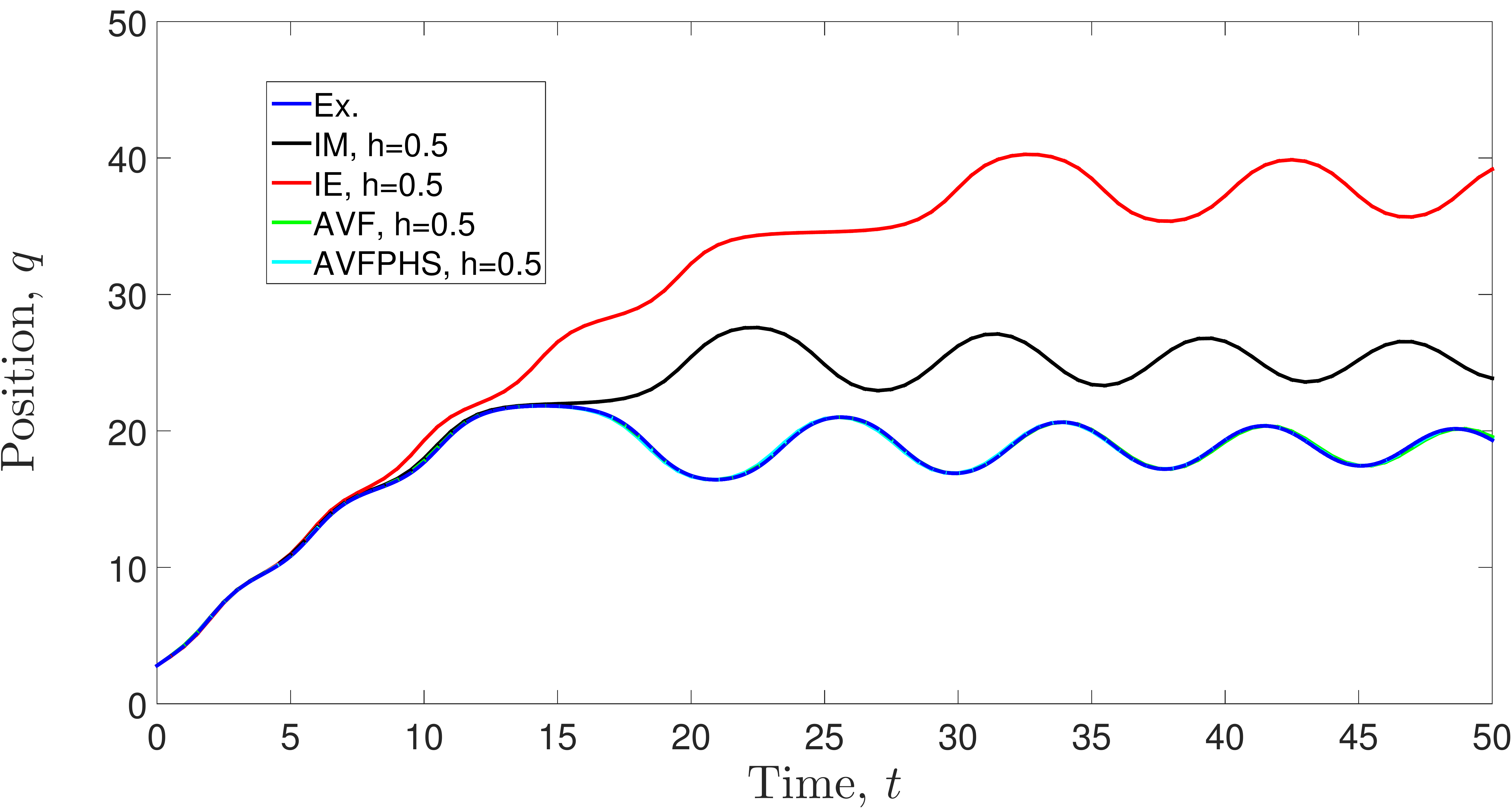}
\end{center}
\caption{Evolution of the pendulum angle $q$ for \eqref{pend}. Shown are the exact solution (Ex.), the method \eqref{avfPHS} (AVFPHS), the implicit midpoint method (IM), the averaged vector field method (AVF), and the improved Euler method (IE). \label{pendAng}}
\end{figure}

\begin{figure}[!t]
 \begin{center}
\includegraphics[width=1\linewidth]{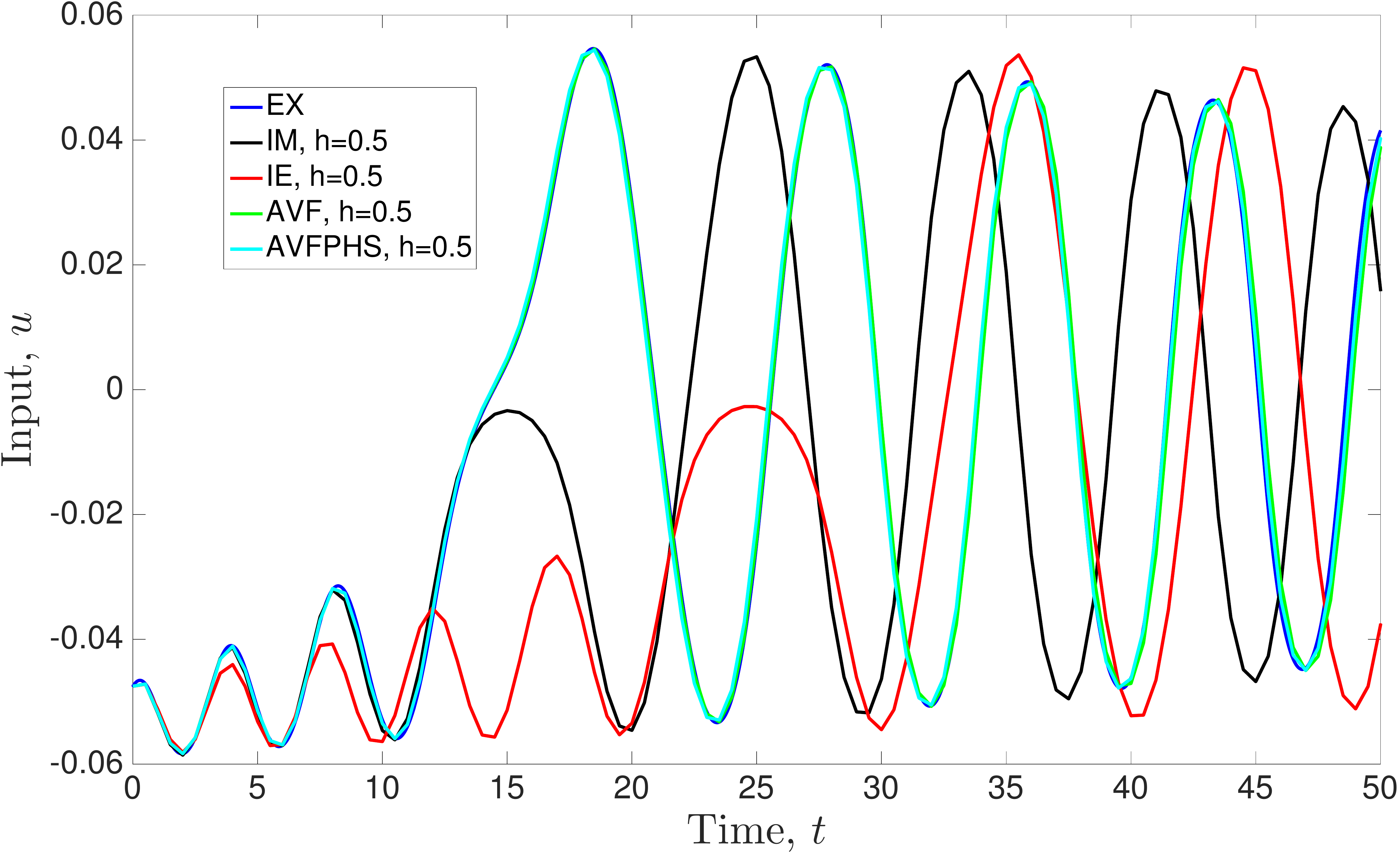}
\end{center}
\caption{Evolution of the input $u$ for \eqref{pend}. Shown are the exact solution (Ex.), the method \eqref{avfPHS} (AVFPHS), the implicit midpoint method (IM), the averaged vector field method (AVF), and the improved Euler method (IE). \label{pendU}}
\end{figure}

\begin{figure}[!t]
 \begin{center}
\includegraphics[width=1\linewidth]{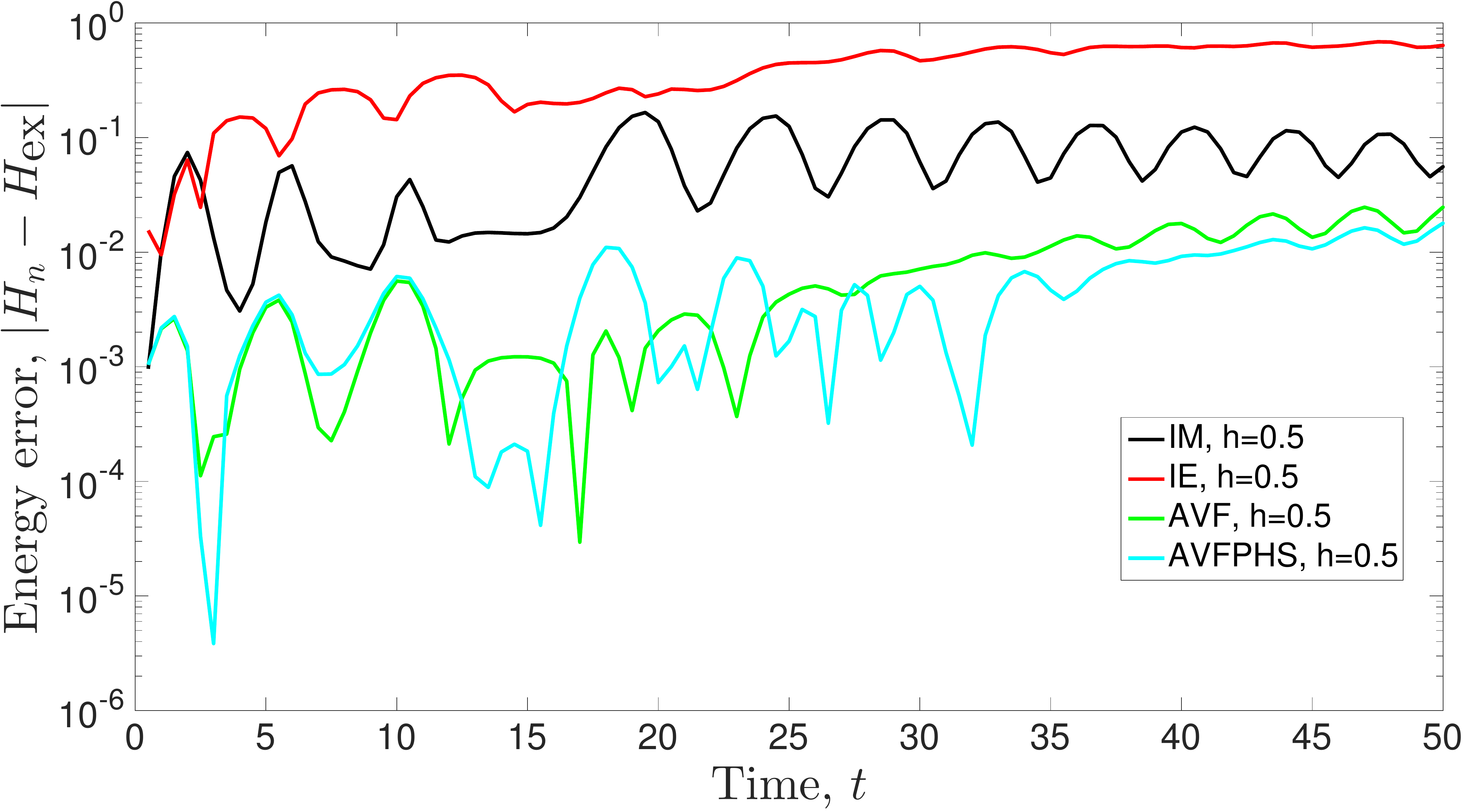}
\end{center}
\caption{Evolution of the absolute error in the Hamiltonian for \eqref{pend}. Shown are the method \eqref{avfPHS} (AVFPHS), the implicit midpoint method (IM), the averaged vector field method (AVF), and the improved Euler method (IE). \label{pendHam}}
\end{figure}
\subsection{Capacitor Microphone}
For a system with dissipation, we consider the capacitor microphone from \cite{vanderschaft}, which can also be written on the format \eqref{PHS} with $x = [q,p,Q]^T \in \mathbb{R}^3$ and
\begin{IEEEeqnarray}{rCl}
H(x) &=& \frac{1}{2m}p^2+\frac{1}{2}(q-\bar{q})^2+\frac{1}{2}qQ^2, \nonumber \\
B(x) &=& \left[\begin{array}{ccc}
0 & 1 & 0\\
-1 & -c & 0\\
0 & 0 & -1/R\\
\end{array} \right], \nonumber  \\
G(x) &=& [0,1,1/R]^T, \label{mic}  \\
u &=& -\phi(y) = -\frac{1}{2}\sqrt[3]{y}. \nonumber 
\end{IEEEeqnarray}
Here $R = 100$ is the resistance, and $c = 0.1$ the damping constant of the spring to which the right capacitor plate, with mass $m=4$, is attached. $\bar{q} = 3$ is the equilibrium point of the spring. 

In Fig. \ref{micFig} the evolution of the absolute error in the Hamiltonian for the dissipative system \eqref{mic} is compared for method \eqref{avfPHS} and the improved Euler method, with step size $h=0.5$. Method \eqref{avfPHS} is again seen to better capture the correct evolution of the energy when compared to improved Euler. Results for the implicit midpoint method and the averaged vector field method were here comparable to method \eqref{avfPHS}.

\begin{figure}[!t]
 \begin{center}
\includegraphics[width=1\linewidth]{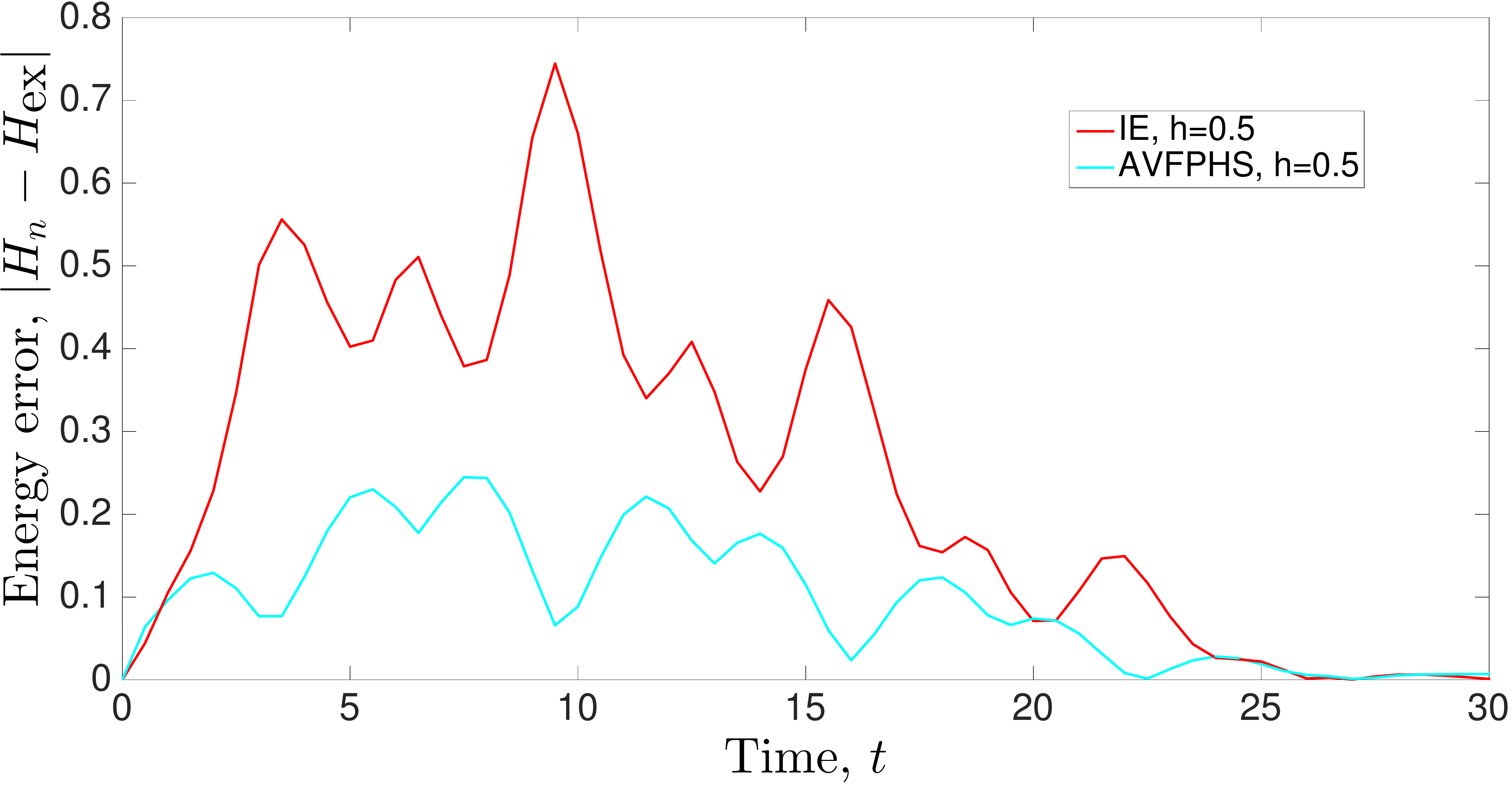}
\end{center}
\caption{Evolution of the absolute error in the Hamiltonian $|H(x_n)-H(x(t_n)|$ for \eqref{mic}. Shown are the method \eqref{avfPHS} (AVFPHS), and the improved Euler method (IE). \label{micFig}}
\end{figure}

\section{Conclusion}
\label{concluding}
We have presented a systematic way to design discrete port-Hamiltonian systems, starting from a continuous system, by applying discrete gradient methods and splitting methods. The obtained discrete port-Hamiltonian systems are passive, and can be globally stabilized with respect to an equilibrium point by an appropriate choice of the input of the discrete system. The obtained discrete systems are port-Hamiltonian in the sense that they preserve a discrete notion of passivity and a generalized Dirac structure. The methods derived using this approach showed promising results in numerical experiments.


%

\appendix[Higher Order Discrete Gradient Methods]
Generalisations of the second order method (\ref{avfPHS}) to higher order can be obtained using a collocation idea as in \cite{Hairer2011}. Let us denote the Lagrange basis function on the node $c_j$ by $\ell_j$ for $j = 1,\dots,s$, write $\sigma(t)$ for the collocation polynomial, and define $X_{\tau}:=\sigma(t_n+\tau h)$ and $X_j:=\sigma(t_n+c_j h)$. Now, consider the following collocation method to integrate \eqref{PHS}:
\begin{IEEEeqnarray}{rCl}
\sigma(t_n)&=&x_n, \nonumber\\*
\dot{\sigma}(t_n+c_jh)&=&B(X_j)\bar{\nabla} H_j+G(X_j)u_j, \label{coll}\\*
x_{n+1}&=&\sigma(t_n+h),\nonumber
\end{IEEEeqnarray}

with
\begin{IEEEeqnarray*}{rCl}
b_j &:=& \int_0^1 \ell_j(\alpha) \, d\alpha, \\ 
\bar{\nabla} H_j&:=&\int_0^1\, \frac{\ell_j(\alpha)}{b_j}\,\nabla H(X_{\alpha})\, d\alpha,\ \ y_j := G(X_j)^T\bar{\nabla} H_j,
\end{IEEEeqnarray*}
and where the discrete controls $u_j$ depend on $\sigma$.
Using Lagrange interpolation we can express the derivative of the collocation polynomial as
$$
\dot{\sigma}(t_n+\tau h)=\sum_{j=1}^s\ell_j(\tau)\left[  B(X_j) \bar{\nabla} H_j+G(X_j)u_j\right],
$$
and obtain $X_{\tau}=\sigma(t_n+\tau h)$ by integration. We may define stage values for the output as
$$
y_{j} = G(X_j)^T\bar{\nabla} H_j.
$$
The collocation polynomial $X_{\tau}$ gives a natural continuous form of the numerical solution on the whole interval of integration. Along the approximated solution  $X_{\tau}$, a polynomial of degree $s$, we will show that the numerical method preserves a discrete passivity property, for quadratures $c_1,\dots ,c_s$  with positive weights $b_1,\dots ,b_s$.

In fact we have
$$
H(x_{n+1})-H(x_n)=h\int_0^1\nabla H(X_{\tau})^T\dot{\sigma}(t_n+\tau h)\,d\tau.
$$
Then after simple calculations we obtain
\begin{IEEEeqnarray*}{rCl}
H(x_{n+1})-H(x_n)&=&h\sum_{j=1}^sb_j\bar{\nabla} H_j^TB(X_j)\bar{\nabla} H_j \\*
&& +\> h\sum_{j=1}^sb_j\bar{\nabla} H_j^TG(X_j)u_j,
\end{IEEEeqnarray*}
and upon using that $B(\cdot )$ is skew-symmetric
\begin{IEEEeqnarray}{rCl}
H(x_{n+1})-H(x_n)&=& h\sum_{j=1}^sb_j\bar{\nabla} H_j^TG(X_j)u_j \nonumber \\*
&=&  h\sum_{j=1}^sb_jy_j^Tu_j, \label{CollEn}
\end{IEEEeqnarray}
which clearly satisfies the discrete energy balance equation \eqref{DEC}. The consequent passivity inequality becomes
$$
H(x_{n+1})-H(x_n)\le h\sum_{j=1}^sb_jy_j^Tu_j.
$$
We now state and prove the following theorem
\begin{theorem}
\label{th:pass}
Assume that the discrete port-Hamiltonian system \eqref{coll} defines a unique one step map $x_{n+1} = \Psi(x_n)$, and that the system is passive with radially unbounded positive definite storage function $H$. Furthermore assume that no numerical solution $\{x_n\}_{n \in \mathbb{N}}$ satisfying the system of equations
\begin{IEEEeqnarray}{rCl}
\sigma(t_n)&=&x_n, \nonumber\\
\dot{\sigma}(t_n+c_jh)&=&B(X_j)\bar{\nabla} H_j, \label{collSimp}\\
x_{n+1}&=&\sigma(t_n+h),\nonumber
\end{IEEEeqnarray}
can simultaneously satisfy the requirement $y_j = 0$ for $j = 1,\dots,s$ at every solution step, other than the trivial solution, i.e. $x_n = 0$ for $n \in \mathbb{N}$.
Then the origin $x=0$ can be globally stabilised with the choice of an appropriate control input 
$u_j=-\phi(y_j)$ where $\phi$ is a function such that $\phi(0)=0$ and $y^T\phi(y)>0$ for all $y\neq 0$.
\end{theorem}
\begin{proof}
From \eqref{CollEn} we have
\begin{IEEEeqnarray*}{rCl}
\Delta H(x_n) &=& H(\Psi(x_n))-H(x_n) \\*
&=& -h\sum_{j=1}^sb_j\bar{\nabla}H_j^TG(X_j)\phi(G(X_j)^T\bar{\nabla}H_j) \leq 0,
\end{IEEEeqnarray*}
where the last inequality follows from the weights $b_j$ being positive and the properties of $\phi$. Note that this inequality holds termwise. Since $H$ is continuous on $\mathbb{R}^n$ it is a (discrete) Lyapunov function for \eqref{coll} on $\mathbb{R}^n$. In addition, because $H$ is radially unbounded, it follows that all solutions of this discrete system are bounded.

Consequently from Proposition \ref{invariance}, $x_n \rightarrow M \neq \emptyset$, where $M$ is the largest positively invariant set contained in the set $E = \{x \in \mathbb{R}^n: \Delta H(x) = 0\}$. Thus if $M = \{0\}$ the origin will be globally asymptotically stable.

Now, from the above calculations $\Delta H(x_n) = 0$ implies $y_j = 0$ for all $j$, which from the properties of $\phi$ implies that $u_j = 0$ for all $j$. Therefore \eqref{coll} reduces to \eqref{collSimp}. Now from the zero-state observability requirement $x_n = 0$ for $n \in \mathbb{N}$, and consequently $M = \{0\}$.
\end{proof}


\ifCLASSOPTIONcaptionsoff
  \newpage
\fi



%
\bibliographystyle{IEEEtran}
\bibliography{./References}
%

\begin{IEEEbiographynophoto}{Elena Celledoni}
Biography will be included in the final published version, if the paper is accepted.
\end{IEEEbiographynophoto}

\begin{IEEEbiographynophoto}{Eirik Hoel H{\o}iseth}
Biography will be included in the final published version, if the paper is accepted.
\end{IEEEbiographynophoto}

\end{document}